\documentclass[12pt,amsfonts]{article}
\usepackage[margin=1in]{geometry}  
\usepackage{graphicx}              
\usepackage[latin1]{inputenc}
\usepackage{amsmath, amssymb, amsthm, epsfig}               
\usepackage{enumerate}
\usepackage{amsfonts}              
\usepackage{amsthm}                
\usepackage{amsthm}
\usepackage{enumerate}
\theoremstyle{plain}
\usepackage{color}
\def\nd{\noindent}
\def\thend{\rule{3mm}{3mm}}

\newtheorem{claim}{Claim}[section]
\newtheorem{theorem}{Theorem}[section]

\newtheorem{lemma}{Lemma}[section]

\newtheorem{corollary}{Corollary}[section]
\newtheorem*{theorem*}{Theorem}

\numberwithin{equation}{section}

\begin{document}
\title{Normalized solutions for a Schr\"{o}dinger equation with critical growth in $\mathbb{R}^{N}$}
\author{ Claudianor O. Alves,\footnote{C.O. Alves was partially supported by CNPq/Brazil   grant 304804/2017-7.} \,\, Chao Ji\footnote{Corresponding author} \footnote{C. Ji was partially supported by Natural Science Foundation of Shanghai(20ZR1413900,18ZR1409100).} \,\, and \,\,   Olimpio H. Miyagaki\footnote{O. H. Miyagaki was supported by FAPESP/Brazil grant 2019/24901-3  and CNPq/Brazil grant 307061/2018-3.}}

\maketitle

\begin{abstract}
In this paper  we study the existence of normalized solutions to the following nonlinear Schr\"{o}dinger equation with critical growth
\begin{align*}
 \left\{
\begin{aligned}
&-\Delta u=\lambda u+f(u), \quad
\quad
\hbox{in }\mathbb{R}^N,\\
&\int_{\mathbb{R}^{N}}|u|^{2}dx=a^{2},
\end{aligned}
\right.
\end{align*}
where $a>0$, $\lambda\in \mathbb{R}$ and $f$ has an exponential critical growth when $N=2$, and  $f(t)=\mu |u|^{q-2}u+|u|^{2^*-2}u$ with $q \in (2+\frac{4}{N},2^*)$, $\mu>0$ and $2^*=\frac{2N}{N-2}$ when $N \geq 3$. Our main results complement some recent results for $N \geq 3$ and it is totally new for $N=2$.
\end{abstract}



\section{Introduction}

This paper concerns the existence of normalized solutions to the following nonlinear Schr\"{o}dinger equation with critical growth
\begin{align}\label{11}
 \left\{
\begin{aligned}
	&-\Delta u=\lambda u+f(u), \quad
	\quad
	\hbox{in }\mathbb{R}^N,\\
	&\int_{\mathbb{R}^{N}}|u|^{2}dx=a^{2},
\end{aligned}
\right.
\end{align}
where $a>0$, $\lambda\in \mathbb{R}$ and $f$ has an exponential critical growth when $N=2$, and  $f(t)=\mu |u|^{q-2}u+|u|^{2^*-2}u$ with $q \in (2+\frac{4}{N},2^*)$, $\mu>0$ and 
$2^*=\frac{2N}{N-2}$ when $N \geq 3$.\\

The equation \eqref{11} arises when ones look for the solutions with prescribed mass for the nonlinear Schr\"{o}dinger equation
\begin{equation*}
i\frac{\partial \psi}{\partial t}+\triangle \psi+h(|\psi|^{2})\psi=0
\quad
\hbox{in }\mathbb{R}^N.
\end{equation*}

A stationary wave solution is a solution of the form $\psi(t, x)=e^{-i\lambda t}u(x)$, where $\lambda\in \mathbb{R}$ and $u:\mathbb{R}^N\rightarrow \mathbb{R}$ is a time-independent that must solve the elliptic problem
\begin{equation} \label{g00}
-\Delta u=\lambda u+g(u), \quad \mbox{in} \quad \mathbb{R}^N,
\end{equation}
where $g(u)=h(|u|^2)u$. For some values of $\lambda$ the existence of nontrivial solutions for (\ref{g00}) are obtained as the critical points of the action functional $J_{\lambda}:H^{1}(\mathbb{R}^N) \to \mathbb{R}$ given by
$$
J_\lambda(u)=\frac{1}{2}\int_{\mathbb{R}^N}(|\nabla u|^2-\lambda |u|^2)\,dx-\int_{\mathbb{R}^N}G(u)\,dx,
$$
where $G(t)=\int_{0}^{t}g(s)\,ds$. In this case the particular attention is devoted to {\it the least action solutions}, namely, the minimizing solutions of $J_\lambda$ among all non-trivial solutions.

Another important way to find the nontrivial solutions for (\ref{g00}) is to search for solutions with {\it prescribed mass}, and in this case $\lambda \in \mathbb{R}$ is part of the unknown. This approach seems to be particularly meaningful from the physical point of view, because there is a conservation of mass.

The present paper has been motivated by a seminal paper due to Jeanjean \cite{jeanjean1} that studied the existence of normalized solutions for a large class of Schr\"{o}dinger    equations of the type
\begin{align}\label{pjeanjean}
	\left\{
	\begin{aligned}
		&-\Delta u=\lambda u+g(u), \quad
		\quad
		\hbox{in }\mathbb{R}^N,\\
		& \int_{\mathbb{R}^{N}}|u|^{2}dx=a^{2},
	\end{aligned}
	\right.
\end{align}
with $N \geq 2$, where function $g:\mathbb{R} \to \mathbb{R}$ is an odd continuous function with subcritical growth that satisfies some technical conditions. One of these conditions is the following: $\exists (\alpha,\beta) \in \mathbb{R} \times \mathbb{R}$ satisfying
$$
\left\{
\begin{array}{l}
\frac{2N+4}{N}< \alpha \leq \beta < \frac{2N}{N-2}, \quad \mbox{for}\,\, N \geq 3, \\
\mbox{}\\
\frac{2N+4}{N}< \alpha \leq \beta, \quad \mbox{for} \,\, N=1,2,	
\end{array}	
\right.
$$
such that
$$
\alpha G(s) \leq g(s)s \leq \beta G(s) \quad \mbox{with} \quad G(s)=\int_{0}^{s}g(t)\,dt.
$$
As an example of a function $g$ that satisfies the above condition is $g(s)=|s|^{q-2}s$ with $q \in (2+\frac{4}{N},2^*)$ when $N\geq 3$ and $q> 4$ if $N=2$. In order to overcome the loss of compactness of the Sobolev embedding in  whole $\mathbb{R}^N$, the author worked on the space $H_{rad}^{1}(\mathbb{R}^N)$ to get  some compactness results. However the most important and interesting point, in our opinion, is the fact that Jeanjean did not work directly with the energy functional $I:H^{1}(\mathbb{R}^N) \to \mathbb{R}$ associated with the problem \eqref{pjeanjean} given by

$$
I(u)=\frac{1}{2}\int_{\mathbb{R}^{N}} |\nabla u|^2 dx-\int_{\mathbb{R}^{N}} G(u)dx.
$$
In his approach, he considered the functional $\widetilde{I}:H^{1}(\mathbb{R}^N) \times \mathbb{R}\to \mathbb{R}$ given by
$$
\widetilde{I}(u,s)=\frac{e^{2s}}{2}\int_{\mathbb{R}^{N}} |\nabla u|^2 dx-\frac{1}{e^{Ns}}\int_{\mathbb{R}^{N}} G(e^{\frac{Ns}{2}}u(x))\,dx.
$$
After a careful analysis, it was proved that $I$ and $\widetilde{I}$ satisfy the mountain pass geometry on the manifold
$$
S(a)=\{u \in H^{1,2}(\mathbb{R}^N)\,:\, | u |_2=a\, \},
$$
and their mountain pass levels are equal, 
which we denote by $\gamma(a)$. Moreover, using the properties of $\widetilde{I}(u,s)$, it  was obtained a $(PS)$ sequence $(u_n)$ to $I$ associated with the mountain pass level $\gamma(a)$  which is bounded in $H_{rad}^{1}(\mathbb{R}^N)$. Finally, after some estimates, the author was able to prove that the weak limit of $(u_n)$, denoted by $u$, is nontrivial, $u \in S(a)$ and  it verifies
$$
	-\Delta u-g(u)=\lambda_au \quad \mbox{in} \quad \mathbb{R}^N,
$$
for some $\lambda_a<0$. An example of a nonlinearity explored in \cite{jeanjean1} we cite
$$
g(t)=\mu |t|^{q-2}t\quad t \in \mathbb{R},  \leqno{(g_0)}
$$
where $\mu>0$ and $q \in (2+\frac{4}{N},2^*)$.
We recall that   the study of the normalized problem despite being  more convenient in the application, this brings some difficulties such as Nehari manifold method can not be applied  because the constant $\lambda_a$ is unknown  in the problem;  it is necessary to prove that  the weak limit  belongs  to the constrained manifold; and   also it  brings some difficult to apply some usual approach for obtaining the boundedness of the Palais Smale sequence.

We recall that the number $\bar{q}:=2+\frac{4}{N}$ is called in the literature as $L^2$-critical exponent, which come from   Gagliardo-Nirenberg inequality (see \cite[Theorem 1.3.7, page 9]{CazenaveLivro}. If $g$ is of the form $(g_0)$ with $q \in (2,\bar{q})$, we say that the problem is $L^2$-subcritical, while in the case $q \in (\bar{q},2^*)$ the problem is $L^2$-supercritical. Associated with the $L^2$-supercritical problem, we would like to cite \cite{Bartschmolle}, where the authors studied a problem involving vanishing potential. In the purely $L^2$-critical case, that is, $q =2+\frac{4}{N}$, related problems were studied in \cite{Cheng,Miao}.

In \cite{Nicola1},  Soave  studied the normalized solutions for the nonlinear Schr\"{o}dinger equation (\ref{11}) with combined power nonlinearities of the type
$$
f(t)=\mu |t|^{q-2}t+|t|^{p-2}t,\quad t \in \mathbb{R},  \leqno{(f_0)}
$$
where
$$
2<q\leq 2+\frac{4}{N}\leq p<2^{*},\,\, p\neq q \quad \text{and}\,\, \mu\in \mathbb{R}.
$$
He showed that interplay between subcritical, critical and supercritical nonlinearities strongly affects the geometry of the functional as well as the  existence and properties of ground states.

Recently some authors have considered the problem (\ref{11}) with $f$ of the form $(f_0)$ but with $p=2^*$, which implies that $f$ has a critical growth in the Sobolev sense. In \cite{JeanjeanJendrejLeVisciglia}, the existence of a ground state normalized solution is obtained as minimizer of the constrained functional assuming that $q \in (2,2+\frac{4}{N})$. While in \cite{JeanjeanLe} a multiplicity result was established, where the second solution is not a ground state. For the general case $q \in (2,2^*)$, we would like to mention Soave \cite{Nicola2}, where the existence result was obtained by imposing that $\mu a^{(1-\gamma_p)q}<\alpha$, where $\alpha$ is a specific constant that depends on $N$ and $q$ and $\gamma_p=\frac{N(p-2)}{2p}$.  We have seen  that the results in the paper are deeply dependent on the assumptions about $a$ and $\mu$, because by the Pohozaev identidy the problem (\ref{11}) does not have any solution if $\lambda =- a \geq 0$  and $\mu >0.$ Still related to the case $q \in (2+\frac{4}{N},2^*)$, we would like to refer \cite{AkahoriIbrahimKikuchiNawa1,AkahoriIbrahimKikuchiNawa2} where the existence of least action solutions was proved with $\mu>0$ and $N\geq 4,$ and for $N=3$ by supposing a technical on the constants $\lambda$ and $\mu$.

We recall that elliptic problems involving critical Sobolev exponent were studied by many researchers after appeared the pioneering paper  by Brezis and Nirenberg \cite{BN}, which have had many progresses in several directions. We would like to mention  the  excellent book \cite{Willem}, for a review on this subject. In our setting, since  if $\lambda=0$, the problem (\ref{11}) does not have any solution for any $ \mu$, then  taking $\lambda$ as  a Lagrange multiplier, it is able to get solution, by combining the arguments made in \cite{BN} with the concentration compactness principle. For the reader interested in normalized solutions for the Schr\"{o}dinger equations, we would also like to refer  \cite{BartschJeanjeanSaove}, \cite{BartschSaove}, \cite{BellazziniJeanjeanLuo}, \cite{CingolaniJeanjean}, \cite{Guo}, \cite{JeanjeanLu},  \cite{MEDERSKISCHINO},  \cite{BenedettaTavaresVerzini}, \cite{Stefanov},   \cite{TaoVisanZhang},  \cite{WangLi} and references therein.

Our main result for the Sobolev critical case is the following:
\begin{theorem}\label{T1}
Assume that $f$ is of the form $(f_0)$ with $p=2^*$ and $q \in (2+\frac{4}{N},2^*)$. Then, there exists $\mu^*=\mu^*(a)>0$ such that  the problem \eqref{11} admits a couple $(u_{a}, \lambda_{a})\in H^{1}(\mathbb{R}^N)\times \mathbb{R}$ of weak solutions such that $\int_{\mathbb{R}^{N}}|u|^{2}dx=a^{2}$ and $\lambda_{a}<0$ for all $\mu \geq \mu^*$.
\end{theorem}

The above theorem complements the results found in \cite{Nicola2} for the $L^2$-supercritical case, because in that paper $\mu \in (0,a^{-(1-\gamma_p)q}\alpha)$ for some $\alpha>0$, then $\mu$ cannot be large enough, while in our paper $\mu$ can be large enough, because $\mu \in [\mu^*(a),+\infty)$.  Here, we used a different approach from that explored in \cite{Nicola2}, because we work directly with the mountain pass geometry and concentration-compactness principle due to Lions \cite{Lions}, while in \cite{Nicola2}, Soave employed minimization technique and used the properties of the Pohozaev manifold.

Motivated by the research made in the critical Sobolev case, in this paper we also study the exponential critical growth  for $N=2$, which is a novelty for this type of problem. To the best our knowledge we have not  found any reference involving   normalizing problem  with the exponential critical growth. We recall that in $\mathbb{R}^2$, the natural growth restriction on the function $f$  is given by the inequality
of Trudinger and Moser \cite{M,T}. More precisely, we say that a
function $f$ has an exponential critical growth if there is $\alpha_0 >0$ such that
$$ \lim_{|s| \to \infty} \frac{|f(s)|}{e^{\alpha s^{2}}}=0
\,\,\, \forall\, \alpha > \alpha_{0}\quad \mbox{and} \quad
\lim_{|s| \to \infty} \frac{|f(s)|}{e^{\alpha s^{2}}}=+ \infty
\,\,\, \forall\, \alpha < \alpha_{0}.
$$
We would like to mention that problems involving exponential critical growth have received a special attention at last years, see for example, \cite{A, AdoOM, ASS1, Montenegro, Cao,DMR,DdOR, OS,doORuf} for semilinear elliptic equations, and \cite{1,AlvesGio,2,5} for quasilinear equations.

In this case, we assume that $f$ is a continuous function that satisfies the following conditions:

\begin{itemize}
	\item[\rm ($f_1$)]$\displaystyle \lim_{t \to 0}\frac{|f(t)|}{|t|^{\tau}}=0$ as $t\rightarrow 0$,\, \mbox{for some}\, $\tau>3$;

	\item[\rm ($f_2$)]$$
	\lim_{|t|\rightarrow +\infty} \frac{|f(t)|}{e^{\alpha t^{2}}}
	=
	\begin{cases}
		0,& \hbox{for } \alpha> 4\pi,\\
		+\infty,& \hbox{for }  0<\alpha<4\pi;
	\end{cases}
	$$

	\item[\rm ($f_3$)] there exists a positive constant $\theta>4$ such that
	\begin{equation*}
		0<\theta F(t)\leq tf(t), \, \, \forall\,  t \not= 0,
		\,
		\hbox{where }
		F(t)=\int_{0}^{t}f(s)ds;
	\end{equation*}
	
	\item[\rm ($f_4$)]  there exist constants $p>4$ and $\mu>0$ such that
	\begin{equation*}
		sgn(t)f(t)\geq \mu  \, |t|^{p-1}\quad \text{for all}\,\, t \not=0,
	\end{equation*}
where $sgn:\mathbb{R}\setminus \{0\} \to \mathbb{R}$ is given by
$$
sgn(t)=
\left\{
\begin{array}{l}
1, \quad \mbox{if} \quad t>0
\mbox{}\\
-1, \quad \mbox{if} \quad t<0.	
\end{array}
\right.
$$

\end{itemize}

Our main result is as follows:
\begin{theorem}\label{T2}
	Assume that $f$ satisfies $(f_1)-(f_4)$. If $a \in (0,1)$, then there exists $\mu^*=\mu^*(a)>0$ such that  the problem \eqref{11} admits a couple $(u_{a}, \lambda_{a})\in H^{1}(\mathbb{R}^2)\times \mathbb{R}$ of weak solutions with $\int_{\mathbb{R}^{2}}|u|^{2}dx=a^{2}$ and $\lambda_{a}<0$ for all $\mu \geq \mu^*$.
\end{theorem}

 In the proof of Theorem \ref{T1} and Theorem \ref{T2} we borrow the ideas developed in Jeanjean \cite{jeanjean1}. The main difficulty in the proof of these theorems is associated with the fact that we are working with critical nonlinearities in whole $\mathbb{R}^N$. As above mentioned, in the proof of Theorem \ref{T1}, the concentration-compactness principle due to Lions \cite{Lions} is crucial in our arguments, while in the proof of Theorem \ref{T2}, the Trundiger-Moser inequality developed by Cao \cite{Cao} plays an important role in a lot of estimates. Moreover, in the proofs of  these theorems we shall work on the space $H^{1}_{rad}(\mathbb{R}^N)$, because it has very nice compact embeedings. Moreover, by Palais'  principle of symmetric criticality, see \cite{palais}, it is well known that the solutions in $H^{1}_{rad}(\mathbb{R}^N)$ are in fact solutions in whole $H^{1}(\mathbb{R}^N)$.

\vspace{0.5 cm}

\noindent \textbf{Notation:} From now on in this paper, otherwise mentioned, we use the following notations:
\begin{itemize}
	\item $B_r(u)$ is an open ball centered at $u$ with radius $r>0$, $B_r=B_r(0)$.

	\item   $C,C_1,C_2,...$ denote any positive constant, whose value is not relevant.
	
	\item  $|\,\,\,|_p$ denotes the usual norm of the Lebesgue space $L^{p}(\mathbb{R}^N)$, for $p \in [1,+\infty]$,
    $\Vert\,\,\,\Vert$ denotes the usual norm of the Sobolev space $H^{1}(\mathbb{R}^N)$.

	\item $o_{n}(1)$ denotes a real sequence with $o_{n}(1)\to 0$ as $n \to +\infty$.

\end{itemize}

\section{Normalized solutions: The Sobolev critical case for $N\geq 3$}

In order to follow the same strategy in \cite{jeanjean1}, we need the following definitions to introduce our variational procedure.

\begin{itemize}
	\item[\rm (1)] $S(a)=\{u \in H^{1}(\mathbb{R}^N)\,:\, | u |_2=a\, \}$ is the sphere of radius $a>0$ defined with the norm $|\,\,\,\,|_2$.

	\item[\rm (2)] $J: H^{1}(\mathbb{R}^N)\rightarrow \mathbb{R}$ with
 $$
	J(u)=\frac{1}{2}\int_{\mathbb{R}^{N}} |\nabla u|^2 dx-\int_{\mathbb{R}^{N}} F(u)dx,
$$
where
$$
F(t)=\frac{\mu}{q}|t|^q+\frac{1}{2^*}|t|^{2^*}, \quad t \in \mathbb{R}.
$$
Hereafter,  $H=H^{1}(\mathbb{R}^N)\times \mathbb{R}$ is equipped with the scalar product
$$
\langle\cdot, \cdot\rangle_{H}=\langle\cdot, \cdot\rangle_{H^{1}(\mathbb{R}^N)}+\langle\cdot, \cdot\rangle_{\mathbb{R}}
$$
and corresponding norm
$$
\Vert \cdot\Vert_{H}=(\Vert \cdot\Vert_{H}^{2}+ \vert \cdot\vert_{ \mathbb{R}}^{2})^{1/2}.
$$
In this section, $f$ denotes the function $f(t)=\mu|t|^{q-2}t+|t|^{2^*-2}t$ with $t \in\mathbb{R}$, and so, $F(t)=\int_0^tf(s)\,ds$.
    \item[\rm (3)]$\mathcal{H}: H\rightarrow H^{1}(\mathbb{R}^N)$ with
	\begin{equation*}
	\mathcal{H}(u, s)(x)=e^{\frac{Ns}{2}}u(e^{s}x).
\end{equation*}

  \item[\rm (4)] $\tilde{J}: H\rightarrow \mathbb{R}$ with
  $$
	\tilde{J}(u, s)=\frac{e^{2s}}{2}\int_{\mathbb{R}^{N}} |\nabla u|^2 dx-\frac{1}{e^{Ns}}\int_{\mathbb{R}^{N}} F(e^{\frac{Ns}{2}}u(x))\,dx
$$
or
$$
\tilde{J}(u, s)=\frac{1}{2}\int_{\mathbb{R}^{N}} |\nabla v|^2 dx-\int_{\mathbb{R}^{N}} F(v(x))\,dx=J(v)\quad  \mbox{for}\ v=\mathcal{H}(u, s)(x).
$$
\end{itemize}

Throughout this section, $S$ denotes the following constant
\begin{equation} \label{Sp*}
	S=
	\inf_{ {\footnotesize{
				\begin{array}{l}
					u \in D^{1,2}(\mathbb{R}^N) \\
					u \not=0
	\end{array}}}}
	\frac{ \int_{\mathbb{R}^N}|\nabla u|^{2}\,dx}{\left(\int_{\mathbb{R}^N}|u|^{2^*}\,dx\right)^{\frac{2}{2^*}}},
\end{equation}
where $2^*=\frac{2N}{N-2}$ for $N\geq 3$, and $D^{1,2}(\mathbb{R}^N)$ is the Banach space given by
$$
D^{1,2}(\mathbb{R}^N)=\left\{u \in L^{2^*}(\mathbb{R}^N)\;:\;|\nabla u|^{2} \in L^{2}(\mathbb{R}^N)\right\}
$$
endowed with the norm
$$
\|u\|_{D^{1,2}(\mathbb{R}^N)}=\left(\int_{\mathbb{R}^N}|\nabla u|^{2}\,dx\right)^{\frac{1}{2}}.
$$
It is well known that the embedding $D^{1,2}(\mathbb{R}^N) \hookrightarrow L^{2^*}(\mathbb{R}^N)$ is continuous.

\subsection{The minimax approach}

We shall prove that $\tilde{J}$ on $S(a)\times \mathbb{R}$  possesses a kind of mountain-pass geometrical structure.

\begin{lemma}\label{vicentej} Let $u\in S(a)$ be arbitrary but fixed. Then we have:\\
\noindent (i)  $\vert \nabla \mathcal{H}(u, s) \vert_{2}\rightarrow 0$ and $J(\mathcal{H}(u, s))\rightarrow 0$ as $s\rightarrow -\infty$;\\
\noindent (ii) $\vert \nabla \mathcal{H}(u, s) \vert_{2}\rightarrow +\infty$ and $J(\mathcal{H}(u, s))\rightarrow -\infty$ as $s\rightarrow +\infty$.
\end{lemma}

\begin{proof} \mbox{} By a straightforward calculation, it follows that
\begin{equation} \label{CONV0}\int_{\mathbb{R}^N}\vert \mathcal{H}(u, s)(x)\vert^{2}\,dx=a^{2}, \quad \int_{\mathbb{R}^N}|\mathcal{H}(u, s)(x)|^{\xi}\,dx= e^{\frac{(\xi-2)Ns}{2}}\int_{\mathbb{R}^N}|u(x)|^{\xi}\,dx, \quad \forall \xi \geq 2,
\end{equation}
and
\begin{equation} \label{CONVJ1*}
\int_{\mathbb{R}^N}\vert \nabla \mathcal{H}(u, s)(x)\vert^{2}\,dx=e^{2s}\int_{\mathbb{R}^N}|\nabla u|^2 dx.
\end{equation}
From the above equalities, fixing $\xi>2$, we have
\begin{equation} \label{CONV1}
\vert \nabla \mathcal{H}(u, s) \vert_{2}\rightarrow 0 \quad \mbox{and} \quad |\mathcal{H}(u,s)|_{\xi} \to 0 \quad  \mbox{as} \quad s \to -\infty.
\end{equation}
Hence,
$$
\int_{\mathbb{R}^N}|F( \mathcal{H}(u, s))|\,dx \leq C_1\int_{\mathbb{R}^N}| \mathcal{H}(u, s)|^{q}\,dx+C_2\int_{\mathbb{R}^N}| \mathcal{H}(u, s)|^{2^*}\,dx \to 0 \quad \mbox{as} \quad s \to -\infty,
$$
from where it follows that
$$
J(\mathcal{H}(u, s))\rightarrow 0 \quad \mbox{as} \quad  s\rightarrow -\infty,
$$
showing $(i)$.

In order to show $(ii)$, note that by (\ref{CONVJ1*}),
$$
|\nabla \mathcal{H}(u, s) |_{2}\rightarrow +\infty \quad \mbox{as} \quad s \to +\infty.
$$
On the other hand,
$$
J(\mathcal{H}(u, s)) \leq \frac{1}{2}|\nabla \mathcal{H}(u,s)|_2^2-\frac{\mu}{q}\int_{\mathbb{R}^N}|\mathcal{H}(u,s)|^{q}\,dx=\frac{e^{2s}}{2}\int_{\mathbb{R}^N}|\nabla u|^2 dx-\frac{\mu e^{\frac{(q-2)Ns}{2}}}{q}\int_{\mathbb{R}^N}|u(x)|^{q}\,dx.
$$
Since $q>2+\frac{4}{N}$, the last inequality yields
$$
J(\mathcal{H}(u, s))\rightarrow -\infty \quad \mbox{as} \quad s\rightarrow +\infty.
$$
\end{proof}

\begin{lemma}  \label{PJ1} There exists $K(a)>0$ small enough such that
	$$
	0<\sup_{u\in A} J(u)<\inf_{u\in B} J(u)
	$$
with
$$
A=\left\{u\in S(a), \int_{\mathbb{R}^N} |\nabla u|^2 dx\leq K(a) \right\}\quad \mbox{and} \quad B=\left\{u\in S(a), \int_{\mathbb{R}^N} |\nabla u|^2 dx=2K(a) \right\}.
$$
\end{lemma}
\begin{proof}
We will need the following Gagliardo-Nirenberg inequality: for any $\xi> 2$,
$$
|u|_\xi\leq C(\xi, N)|\nabla u|_2^{\gamma}|u|_2^{1-\gamma},
$$
where $\gamma=N(\frac{1}{2}-\frac{1}{\xi})$. If we fix $|\nabla u|^{2}_{2}\leq K(a)$ and $|\nabla v|^{2}_2=2K(a)$, we derive that
$$
\int_{\mathbb{R}^N}F(u)\,dx \leq C_1|u|^{q}_{q}+C_2|u|^{2^*}_{2^*}.
$$
Then, by the Gagliardo-Nirenberg  inequality,
$$
\int_{\mathbb{R}^N}F(v)\,dx \leq C_1(|\nabla v|_2^2)^{N(\frac{q-2}{4})}+C_2(|\nabla v|_2^2)^{N(\frac{2^*-2}{4})}.
$$
Since  $F(u)\geq 0$ for any $u\in H^{1}(\mathbb{R}^{N})$, we have
\begin{eqnarray*}
J(v)-J(u) &=&\frac{1}{2}\int_{\mathbb{R}^N}|\nabla v|^{2}\,dx-\frac{1}{2}\int_{\mathbb{R}^N}|\nabla u|^{2}\,dx-\int_{\mathbb{R}^N}F(v)\,dx+\int_{\mathbb{R}^N}F(u)\,dx\\
&\geq &\frac{1}{2}\int_{\mathbb{R}^N}|\nabla v|^{2}\,dx-\frac{1}{2}\int_{\mathbb{R}^N}|\nabla u|^{2}\,dx-\int_{\mathbb{R}^N}F(v)\,dx,
\end{eqnarray*}
and so,
$$
J(v)-J(u) \geq \frac{1}{2}K(a)-C_3K(a)^{N(\frac{q-2}{4})}-C_4 K(a)^{N(\frac{2^*-2}{4})}.
$$
Thereby, fixing $K(a)$ small enough of such way that,
$$
\frac{1}{2}K(a)-C_3K(a)^{N(\frac{q-2}{4})}-C_4 K(a)^{N(\frac{2^*-2}{4})}>0,
$$
we get the desired result.
\end{proof}

As a byproduct of the last lemma is the following corollary.
\begin{corollary} There exists $K(a)>0$ such that if $u \in S(a)$ and $|\nabla u|^2_{2}\leq K(a)$, then $J(u) >0$.
\end{corollary}
\begin{proof} Arguing as in the last lemma,
$$
J(u) \geq \frac{1}{2}|\nabla u|_{2}^{2}-C_1|\nabla u|_2^{N(\frac{q-2}{2})}-C_2|\nabla u|_2^{N(\frac{2^*-2}{2})}>0,
$$
for $K(a)$ small enough.	
\end{proof}

In what follows,  we fix $u_0 \in S(a)$ and apply Lemma \ref{vicentej} to get two numbers $s_1<0$ and $s_2>0$, of such way that the functions $u_1=\mathcal{H}(s_1,u_0)$ and $u_2=\mathcal{H}(s_2,u_0)$ satisfy
$$
|\nabla u_1|^2_2<\frac{K(a)}{2}, \,\, |\nabla u_2|_2^2>2K(a),\,\, J(u_1)>0\quad \mbox{and} \quad J(u_2)<0.
$$

Now, following the ideas from Jeanjean \cite{jeanjean1}, we fix the following mountain pass level given by
$$
\gamma_\mu(a)=\inf_{h \in \Gamma}\max_{t \in [0,1]}J(h(t))
$$
where
$$
\Gamma=\left\{h \in C([0,1],S(a)): h(0)=u_1 \,\,\mbox{and} \,\, h(1)=u_2 \right\}.
$$
From Lemma \ref{PJ1},
$$
\max_{t \in [0,1]}J(h(t))>\max \left\{J(u_1),J(u_2)\right\}.
$$

\begin{lemma} \label{ESTMOUNTPASS} There holds $\displaystyle \lim_{\mu \to +\infty}\gamma_\mu(a)=0$.

\end{lemma}
\begin{proof} In what follows we set the path $h_0(t)=\mathcal{H}((1-t)s_1+ts_2,u_0) \in \Gamma$. Then,
$$
\gamma_\mu(a) \leq \max_{t \in [0,1]}J(h_0(t)) \leq \max_{r \geq 0}\left\{\frac{r^2}{2}|\nabla u_0|_{2}^{2}-\frac{\mu}{q}r^{\frac{N(q-2)}{2}}|u_0|_{q}^{q}\right\},
$$
and so, for some positive constant $C_2$,
$$
\gamma_\mu(a) \leq C_2\left(\frac{1}{\mu}\right)^{\frac{4}{N(q-2)-4}} \to 0 \quad \mbox{as} \quad \mu \to +\infty.
$$
Here, we have used the fact that $q>2+\frac{4}{N}$.

\end{proof}

In what follows $(u_n)$ denotes the $(PS)$ sequence associated with the level $\gamma_\mu(a)$, which is  obtained by making $u_n=\mathcal{H}(v_n,s_n)$,  where $(v_n,s_n)$ is the $(PS)$ sequence for $ \tilde{J} $ obtained by \cite[Proposition 2.2 ]{jeanjean1}, associated with the level $\gamma_\mu(a).$ More precisely, we have
\begin{equation} \label{gamma(a)}
J(u_n) \to \gamma_\mu(a) \quad \mbox{as} \quad n \to +\infty,
\end{equation}
and
\begin{equation*} \label{der1}
\|J|'_{S(a)}(u_n)\| \to 0 \quad \mbox{as} \quad n \to +\infty.
\end{equation*}
Setting the functional $\Psi:H^{1}(\mathbb{R}^N) \to \mathbb{R}$ given by
$$
\Psi(u)=\frac{1}{2}\int_{\mathbb{R}^N}|u|^2\,dx,
$$
it follows that $S(a)=\Psi^{-1}(\{a^2/2\})$. Then, by Willem \cite[Proposition 5.12]{Willem}, there exists $(\lambda_n) \subset \mathbb{R}$ such that
$$
||J'(u_n)-\lambda_n\Psi'(u_n)||_{H^{-1}} \to 0 \quad \mbox{as} \quad n \to +\infty.
$$
Hence,
\begin{equation} \label{EQ10}
	-\Delta u_n-f(u_n)=\lambda_nu_n\ + o_n(1) \quad \mbox{in} \quad (H^{1}(\mathbb{R}^N))^*.
\end{equation}
Moreover, another important limit involving the sequence $(u_n)$ is
\begin{equation} \label{EQ1**}
Q(u_n)=\int_{\mathbb{R}^N}|\nabla u_n|^2 dx +N \int_{\mathbb{R}^N} F(u_n) dx- \frac{N}{2}\int_{\mathbb{R}^N} f(u_n) u_n dx\to 0 \quad \mbox{as} \quad n \to +\infty,
\end{equation}
which is obtained using the limit below
$$
{\partial_s}\tilde{J}(v_n,s_n) \to 0 \quad \mbox{as} \quad n \to +\infty,
$$
that was also proved in \cite{jeanjean1}.

Arguing as in \cite[Lemmas 2.3 and 2.4]{jeanjean1}, we know that $(u_n)$ is a bounded sequence, and so, the number $\lambda_n$  must satisfy the equality below
$$
\lambda_n=\frac{1}{|u_n|^{2}_{2}}\left\{ |\nabla u_n|^{2}_{2} -\int_{\mathbb{R}^N} f(u_n) u_n dx \right\}+o_n(1),
$$
or equivalently,
\begin{equation} \label{lambdan}
	\lambda_n=\frac{1}{a^{2}}\left\{ |\nabla u_n|^{2}_{2} -\int_{\mathbb{R}^N} f(u_n) u_n dx \right\}+o_n(1).
\end{equation}

\begin{lemma} \label{Limitacao}There exists $C>0$ such that
$$
\limsup_{n \to +\infty}\int_{\mathbb{R}^N}F(u_n)\,dx \leq C \gamma_\mu(a)
$$
and
$$
\limsup_{n \to +\infty}\int_{\mathbb{R}^N}f(u_n)u_n\,dx \leq C\gamma_\mu(a).
$$
\end{lemma}
\begin{proof} From (\ref{gamma(a)}) and (\ref{EQ1**})
$$
N{J}(u_n)+Q(u_n)=N\gamma_\mu(a)+o_n(1),
$$
then
$$
\frac{N+2}{2}\int_{\mathbb{R}^N}|\nabla u_n|^2 dx - \frac{N}{2}\int_{\mathbb{R}^N} f(u_n) u_n dx=N\gamma_\mu(a)+o_n(1).
$$
Using again (\ref{gamma(a)}), we get
$$
\frac{N+2}{2}\left(2\int_{\mathbb{R}^N} F(u_n) dx+ 2\gamma_\mu(a)+o_n(1)\right) - \frac{N}{2}\int_{\mathbb{R}^N} f(u_n) u_n dx = N\gamma_\mu(a)+o_n(1),
$$
that is,
\begin{equation}\label{Ee}
-(N+2)\int_{\mathbb{R}^N} F(u_n) dx+ \frac{N}{2}\int_{\mathbb{R}^N} f(u_n) u_n dx = 2\gamma_\mu(a)+o_n(1).
\end{equation}
Since $ q \in (2+\frac{4}{N}, 2^{*})$ and  $
F(t)=\frac{\mu}{q}|t|^q+\frac{1}{2^*}|t|^{2^*}, \,\, \forall \, t \in \mathbb{R}^N, $ we obtain
\begin{equation} \label{q}
qF(t) \leq f(t)t, \,\, t \in\mathbb{R}.
\end{equation}
This together with (\ref{Ee}) yields
$$
\left(\frac{qN}{2} -(N+2)\right)\int_{\mathbb{R}^N} F(u_n) dx\leq 2\gamma_\mu(a)+o_n(1),
$$
and so,
$$
\limsup_{n \to +\infty}\int_{\mathbb{R}^N}F(u_n)\,dx \leq C \gamma_\mu(a).
$$
This inequality combined again with  (\ref{Ee}) ensures that
$$
\limsup_{n \to +\infty}\int_{\mathbb{R}^N}f(u_n)u_n\,dx \leq C \gamma_\mu(a).
$$
\end{proof}

\begin{lemma} \label{goodest0} $\displaystyle \limsup_{n \to +\infty}|\nabla u_n|_{2}^{2} \leq C \gamma_\mu(a)$.
	
\end{lemma}
\begin{proof} First of all, let us recall that
	$$
	\int_{\mathbb{R}^N}|\nabla u_n|^2\,dx=2\gamma_\mu(a)+2\int_{\mathbb{R}^N}F(u_n)\,dx+o_n(1).
	$$
	Then, from Lemma \ref{Limitacao},
	$$
	\limsup_{n \to +\infty}|\nabla u_n|_{2}^{2}  \leq (2+C_{1})\gamma_\mu(a).
	$$
\end{proof}

Now, from (\ref{Ee}), the sequence $(\int_{\mathbb{R}^N}F(u_n)\,dx) $ is bounded away from zero, otherwise we would have
$$
\int_{\mathbb{R}^N}F(u_n)\,dx \to 0 \quad \mbox{as} \quad n \to +\infty,
$$
which leads to
$$
\int_{\mathbb{R}^N}f(u_n)u_n\,dx \to 0 \quad \mbox{as} \quad n \to +\infty.
$$
These limits combined with (\ref{Ee}) imply that $\gamma_\mu(a)=0$, which is absurd. From this, in what follows we can assume that
\begin{equation} \label{Fc1}
\int_{\mathbb{R}^N}F(u_n)\,dx\rightarrow C_1>0,\ \mbox{as}\ n \rightarrow \infty.
\end{equation}

\begin{lemma} \label{mu0} The sequence $(\lambda_n)$ is bounded with
$$
\lambda_n = -\frac{\mu}{a^2}\Big(\frac{N}{q}-\frac{N-2}{2}\Big)\int_{\mathbb{R}^N}|u_n|^q \,dx+o_n(1)
$$
and
$$
\limsup_{n \to +\infty}|\lambda_n| \leq \frac{C}{a^2}\gamma_\mu(a),
$$
for some $C>0$.
\end{lemma}
\begin{proof}  The boundedness of $(u_n)$ yields that $(\lambda_n)$ is bounded, because
	\begin{equation}\label{lambda}
	\lambda_n a^2=\lambda_n |u_n|^{2}_{2}=|\nabla u_n|_{2}^{2}-\int_{\mathbb{R}^N}f(u_n)u_n\,dx +o_n(1),
	\end{equation}
	and so,
	\begin{eqnarray*}
	|\lambda_n| &\leq&\frac{1}{a^2}\left( |\nabla u_n|_{2}^{2}+\int_{\mathbb{R}^N}f(u_n)u_n\,dx\right) + o_n(1) \\
	&\leq&\frac{C}{a^2}\gamma_\mu(a) + o_n(1).
	\end{eqnarray*}
This guarantees the boundedness of $(\lambda_n)$ and the second inequality is proved.
	
	In order to prove the first inequality, we know by (\ref{EQ1**}) that
	$$
	|\nabla u_n|_{2}^{2}=\frac{N}{2}\int_{\mathbb{R}^N}f(u_n)u_n\,dx  - N\int_{\mathbb{R}^N}F(u_n)\,dx + o_n(1).
	$$
	Inserting this equality in (\ref{lambda}), we obtain
	$$\lambda_n a^2= -\mu\Big(\frac{N}{q}-\frac{N-2}{2}\Big)\int_{\mathbb{R}^N}|u_n|^q \,dx +o_n(1),$$
showing the first inequality.
\end{proof}

In the sequel, we restrict our study to the space $H_{rad}^{1}(\mathbb{R}^N)$.  Then, it is well known that
\begin{equation} \label{convq}
\lim_{n \to +\infty}\int_{\mathbb{R}^N}|u_n|^q\,dx=\int_{\mathbb{R}^N}|u|^q\,dx,
\end{equation}
where $u_n \rightharpoonup u$ in $H_{rad}^{1}(\mathbb{R}^N)$, because $q \in (2+\frac{4}{N},2^*)$.
\begin{lemma} There exists $\mu^*>0$ such that $ u\not=0$ for all $\mu \geq \mu^*>0$.
\end{lemma}	
\begin{proof}
Seeking for a contradiction, let us assume that $u=0$. Then,
\begin{equation} \label{EQq}
\lim_{n \to +\infty}\int_{\mathbb{R}^N}|u_n|^q\,dx=0,
\end{equation}
and by Lemma \ref{mu0},
\begin{equation} \label{ln0}
\limsup_{n \to +\infty} \lambda_n = 0.
\end{equation}
The equality
\begin{equation*}
a^2\lambda_n=|\nabla u_n|^{2}_{2}-\int_{\mathbb{R}^N}f(u_n)u_n\,dx+o_n(1)
\end{equation*}
together with (\ref{EQq}) and (\ref{ln0}) leads to
\begin{equation}  \label{mu2}
|\nabla u_n|_{2}^{2}-|u_n|_{2^*}^{2^*}=o_n(1).
\end{equation}
In what follows, going to a subsequence, we assume that
$$
|\nabla u_n|_{2}^{2}=L+o_n(1) \quad \mbox{and} \quad |u_n|_{2^*}^{2^*}=L+o_n(1).
$$
We claim that $L>0$, otherwise if $L=0$, we must have
\begin{equation}  \label{mu2}
|\nabla u_n|_{2}^{2}=o_n(1),
\end{equation}
then,
$$
|\nabla u_n|^{2}_{2} \to 0,
$$
which is absurd, because $\gamma_\mu(a)>0$.

Since $L>0$, by definition of $S$ in \eqref{Sp*},
$$
S \leq \frac{ \int_{\mathbb{R}^N}|\nabla u_n|^{2}\,dx}{\left(\int_{\mathbb{R}^N}|u_n|^{2^*}\,dx\right)^{\frac{2}{2^*}}}.
$$
Taking the $limsup$ as $n \to +\infty$, we obtain
$$
S \leq \frac{ L}{L^{\frac{2}{2^*}}},
$$
that is,
$$
L \geq S^{\frac{N}{2}}.
$$
On the other hand
$$
o_n(1)+\gamma_\mu(a)-\frac{a^2\lambda_n}{2}=\frac{1}{2}(|\nabla u_n|_{2}^{2}-a^2\lambda_n )-\frac{\mu}{q}|u_n|^{q}_{q}-\frac{1}{2^*}|u_n|^{2^*}_{2^*}=\frac{1}{N}L+o_n(1).
$$
Recalling that $\displaystyle \limsup_{n \to +\infty}|\lambda_n| \leq \frac{C}{a^2}\gamma_\mu(a)$, it follows that
$$
\frac{1}{N}S^{\frac{N}{2}}\leq C\gamma_\mu(a).
$$
Now, fixing $\mu^*$ large enough of a such way that
$$
C\gamma_\mu(a) < \frac{1}{N}S^{\frac{N}{2}}, \quad \forall \mu \geq \mu^*,
$$
we get a new contradiction. This proves that $u\not=0$ for $\mu>0$ large enough.

\end{proof}

\begin{lemma} \label{crescimento} Increasing if necessary $\mu^*$, we have $u_n \to u$ in $L^{2^*}(\mathbb{R}^N)$ for all $\mu \geq \mu^*$.
	
\end{lemma}
\begin{proof}
Using  the concentration-compactness principle due to Lions \cite{Lions},  we can find an at most countable index $\mathcal{J}$, sequences $(x_{i})\subset \mathbb{R}^{N}, ( \kappa_{i}), (\nu_{i})\subset (0, \infty)$ such that\\
\noindent $(i)$	\,\, $|\nabla u_n|^{2} \to \kappa$ weakly-$^*$ in the sense of measure \\
\noindent and \\
\noindent $(ii)$ \,\, $|u_n|^{2^*} \to \nu$ weakly-$^*$ in the sense of measure, \\
and
$$
\left\{
\begin{array}{l}
	(a)\quad  \nu=|u|^{2^*}+\sum_{j \in J}\nu_j \delta_{x_j},\\
	(b)\quad  \kappa \geq |\nabla u|^{2}+\sum_{j \in J}\kappa_j \delta_{x_j},\\
	(c)\quad S \nu_j^{\frac{2}{2^*}} \leq \kappa_j,\,\, \forall j\in \mathcal{J},\\
\end{array}
\right.
$$
where $\delta_{x_{j}}$ is the Dirac mass at the point $x_{j}$. Since
$$
-\Delta u_n-f(u_n)=\lambda_n u_n +o_n(1)\quad \mbox{in} \quad (H^{1}(\mathbb{R}^N))^*,
$$
we derive that
$$
\int_{\mathbb{R}^N}\nabla u_n \nabla \phi \,dx-\lambda_n\int_{\mathbb{R}^N}u_n \phi \,dx=\mu \int_{\mathbb{R}^N}|u_n|^{q-2}u_n\phi\,dx+ \int_{\mathbb{R}^N}|u_n|^{2^*-2}u_n\phi\,dx, \quad \forall \phi \in H^{1}(\mathbb{R}^N).
$$
Now, arguing as in \cite[Lemma 2.3]{GP}, $\mathcal{J}$ is empty or otherwise $\mathcal{J}$ is nonempty but finite. In the case that $\mathcal{J}$ is nonempty but finite, we must have
$$
\kappa_j \geq {S^{\frac{N}{2}}}, \quad \forall j \in \mathcal{J}.
$$
However, by Lemma \ref{goodest0},
$$
 \limsup_{n \to +\infty}|\nabla u_n|_{2}^{2} \leq C \gamma_\mu(a).
$$
Then, if $\mu^*>0$ is fixed of such way that
$$
C \gamma_\mu(a) < \frac{1}{2}{S^{\frac{N}{2}}},
$$
we get a contradiction, and so, $\mathcal{J}= \emptyset$. From this,
\begin{equation} \label{localconvergence}
u_n \to u \quad \mbox{in} \quad  L^{2^*}_{loc}(\mathbb{R}^N).
\end{equation}
\begin{claim} \label{convBR} For each $R>0$, we have
$$
u_n \to u \quad \mbox{in} \quad  L^{2^*}(\mathbb{R}^N \setminus B_R(0)).
$$
\end{claim}	
\noindent Indeed, as $u_n \in H_{rad}^{1}(\mathbb{R}^N)$, we know that
$$
|u_n(x)| \leq \frac{\|u_n\|}{|x|^{\frac{N-1}{2}}}, \quad \mbox{a.e. in} \quad \mathbb{R}^N.
$$
Since $(u_n)$ is a bounded sequence in $H^{1}(\mathbb{R}^N)$, we obtain
$$
|u_n(x)| \leq \frac{C}{|x|^{\frac{N-1}{2}}}, \quad \mbox{a.e. in} \quad \mathbb{R}^N,
$$
and so,
$$
|u_n(x)|^{2^*} \leq \frac{C_1}{|x|^{\frac{N(N-1)}{N-2}}}, \quad \mbox{a.e. in} \quad \mathbb{R}^N.
$$
Recalling that  $\frac{C_1}{|\,\cdot\,|^{\frac{N(N-1)}{N-2}}} \in L^{1}(\mathbb{R}^N \setminus B_R(0))$ and $u_n(x) \to u(x)$ a.e. in $\mathbb{R}^N \setminus B_R(0)$, the Lebesgue's Theorem gives
$$
u_n \to u \quad \mbox{in} \quad  L^{2^*}(\mathbb{R}^N \setminus B_R(0)),
$$
showing the Claim \ref{convBR}. Now, the Claim \ref{convBR} combined with (\ref{localconvergence}) ensures that
$$
u_n \to u \quad \mbox{in} \quad L^{2^*}(\mathbb{R}^N).
$$

\end{proof}

\section{Proof of Theorem \ref{T1}}
From the previous analysis the weak limit $u$ of $(u_n)$ is nontrivial. Therefore, by Lemma \ref{mu0}
$$
\lim_{n \to +\infty} \lambda_n = -\frac{\mu }{a^2}\Big(\frac{N}{q}-\frac{N-2}{2}\Big)\lim_{n \to +\infty}\int_{\mathbb{R}^N}|u_n|^q \,dx=-\frac{\mu }{a^2}\Big(\frac{N}{q}-\frac{N-2}{2}\Big)\int_{\mathbb{R}^N}|u|^q \,dx<0.
$$
So, we can assume without loss of generality that
\begin{equation}\label{new}
\lambda_n \to \lambda_a<0 \quad \mbox{as} \quad n \to +\infty.
\end{equation}
Now, the equality (\ref{EQ10}) and \eqref{new}
imply that
\begin{equation} \label{EQ2}
	-\Delta u-f(u)=\lambda_au, \quad \mbox{in} \quad \mathbb{R}^N.
\end{equation}
Thus,
$$
|\nabla u|_{2}^{2}-\lambda_a|u|_{2}^{2}=\int_{\mathbb{R}^N}f(u)u\,dx.
$$
On the other hand,
$$
|\nabla u_n|_{2}^{2}-\lambda_n|u_n|_{2}^{2}=\int_{\mathbb{R}^N}f(u_n)u_n\,dx+o_n(1),
$$
or yet,
$$
|\nabla u_n|_{2}^{2}-\lambda_a|u_n|_{2}^{2}=\int_{\mathbb{R}^N}f(u_n)u_n\,dx+o_n(1).
$$
Recalling that by Lemma \ref{crescimento}
$$
u_n \to u \quad \mbox{in} \quad L^{2^*}(\mathbb{R}^N),
$$
and using the below limit
$$
u_n \to u \quad \mbox{in} \quad L^{q}(\mathbb{R}^N),
$$
we obtain
$$
\lim_{n \to +\infty}\int_{\mathbb{R}^N}f(u_n)u_n\,dx=\int_{\mathbb{R}^N}f(u)u\,dx,
$$
from where it follows that
$$
\lim_{n \to +\infty}(|\nabla u_n|_{2}^{2}-\lambda_a|u_n|_{2}^{2})=|\nabla u|_{2}^{2}-\lambda_a|u|_{2}^{2}.
$$
Since $\lambda_a<0$, the last equality implies that
$$
u_n \to u \quad \mbox{in} \quad H^{1}(\mathbb{R}^N),
$$
implying that $|u|_{2}^{2}=a$. This establishes the desired result.

\section{Normalized solutions: The exponential critical growth case for $N=2$}

In this section we shall deal with the case $N=2$, where $f$ has an exponential critical growth and $a \in (0,1)$. We start our study recalling that by $(f_1)$ and $(f_2)$, we know that fixed $q >2$, for any $\zeta>0$ and $\alpha>4\pi$, there exists a constant $C>0$, which depends on $q$, $\alpha$, $\zeta$, such that
\begin{equation}
	\label{1.2}
	|f(t)|\leq\zeta |t|^{\tau}+C|t|^{q-1}(e^{\alpha t^{2}}-1) \text{ for all } t \in \mathbb{R}
\end{equation}
and, using $(f_3)$, we have
\begin{equation}
	\label{1.3}
	|F(t)|\leq\zeta |t|^{\tau+1}+C|t|^{q}(e^{\alpha t^{2}}-1) \text{ for all } t \in \mathbb{R}.
\end{equation}
Moreover, it is easy to see that, by \eqref{1.2},
\begin{equation}
	\label{1.4}
	|f(t)t| \leq\zeta |t|^{\tau+1}+C\vert t\vert^{q}(e^{\alpha  t^{2}}-1) \text{ for all } t\in\mathbb{R}.
\end{equation}

Finally, let us recall the following version of Trudinger-Moser inequality as stated e.g. in \cite{Cao}.
\begin{lemma}\label{Cao}
	If $\alpha>0$ and $u\in H^{1}(\mathbb{R}^{2})$, then
	\begin{equation*}
		\int_{\mathbb{R}^{2}}(e^{\alpha  u^{2}}-1)dx<+\infty.
	\end{equation*}
	Moreover, if $|\nabla u|_{2}^{2}\leq 1$, $|  u|_{2}\leq M<+\infty$, and $0<\alpha< 4\pi$, then there exists a positive constant $C(M, \alpha)$, which depends only on $M$ and $\alpha$,  such that
	\begin{equation*}
		\int_{\mathbb{R}^{2}}(e^{\alpha u^{2}}-1)dx\leq C(M, \alpha).
	\end{equation*}
\end{lemma}

\begin{lemma} \label{alphat11} Let $(u_{n})$ be a sequence in $H^{1}(\mathbb{R}^{2})$ with $u_n \in S(a)$ and
	$$
	\limsup_{n \to +\infty} |\nabla u_n |_{2}^{2}  < 1-a^{2}.
	$$
	Then, there exist  $t> 1$, $t$ close to 1,  and $C > 0$  satisfying
	\[
	\int_{\mathbb{R}^{2}}\left(e^{4 \pi |u_n|^{2}} - 1 \right)^t dx \leq C, \,\,\,\,\forall\, n \in \mathbb{N}.
	\]
	
\end{lemma}
\begin{proof}
	As
	\[
	\limsup_{n\rightarrow\infty} |\nabla u_n|_2^2 <1-a^{2} \quad \mbox{and} \quad |u_n|_2^{2}=a^2 <1,
	\]
	there exist $m >0$ and $n_0\in \mathbb{N}$ verifying
	\[
	\|u_n\|^{2} < m < 1,
	\,\,\,\text{for any}\,\,\,  n \geq n_0.
	\]
	Fix  $t > 1$, with $t$ close to $1$ such that $t m < 1$  and
	$$
	\int_{\mathbb{R}^{2}}\left(e^{4 \pi|u_n|^{2}} -1 \right)^t dx
	\leq \int_{\mathbb{R}^{2}}\left(e^{4 t m  \pi(\frac{|u_n|}{||u_n||})^{2}} - 1 \right) dx,\,\,\text{for any}\,\, n \geq n_0,
	$$
	where we have used the inequality
\begin{equation*}
		\label{ineqe}
		(e^{s}-1)^{t}\leq e^{ts}-1, \text{ for }t>1 \text{ and } s\geq 0.
	\end{equation*}
	Hence, by Lemma \ref{Cao}, there exists $C_{1}=C_{1}(t, m, a)>0$
	$$
	\int_{\mathbb{R}^{2}}\left(e^{4 \pi |u_n|^{2}} -1 \right)^t dx \leq C_{1} \, \,\,\,\,\forall\,  n \geq n_0.
	$$
 Now, the lemma follows fixing
	$$
	C=\max\left\{C_1, \int_{\mathbb{R}^{2}}\left(e^{4 \pi |u_{1}|^{2}} -1 \right)^t dx,....,\int_{\mathbb{R}^{2}}\left(e^{4 \pi |u_{n_0}|^{2}} -1 \right)^t dx \right\}.
	$$
\end{proof}

\begin{corollary} \label{Convergencia em limitados} Let $(u_{n})$ be a sequence in $H^{1}(\mathbb{R}^{2})$ with $u_n \in S(a)$ and
	$$
	\limsup_{n \to +\infty} |\nabla u_n |_2^{2}  < 1-a^2.
	$$
	If $u_n \rightharpoonup u$ in $H^{1}(\mathbb{R}^{2})$ and $u_n(x) \to u(x)$ a.e in $\mathbb{R}^{2}$, then
	$$
	F(u_n) \to F(u) \,\, \mbox{in} \,\, L^{1}(B_R(0)), \,\,\text{for any} \,\, R>0.
	$$ 	
\end{corollary}
\begin{proof} By \eqref{1.3}, fixed $q >2$, for any $\zeta>0$ and $\alpha>4\pi$, there exists a constant $C>0$, which depends on $q$, $\alpha$, $\zeta$, such that
\begin{equation*}
	|F(t)|\leq\zeta |t|^{\tau+1}+C|t|^{q}(e^{\alpha t^{2}}-1)\,\, \text{ for all }\, t \in \mathbb{R},
\end{equation*}
from where it follows that,
	\begin{equation} \label{Domina}
		|F(u_n)| \leq \zeta|u_n|^{\tau+1}+C|u_n|^{q}(e^{\alpha  |u_n|^{2}}-1), \quad \forall n \in \mathbb{N}.
	\end{equation}
with
$$
\zeta|u_n|^{\tau+1}+C|u_n|^{q}(e^{\alpha \pi |u_n|^{2}}-1)\rightarrow \zeta |u|^{\tau+1}+C|u|^{q}(e^{\alpha  |u|^{2}}-1)\,\text{ a.e. in }\, \mathbb{R}^2 \text{ as }n\to+\infty.
$$
Similar to arguments in Lemma \ref{alphat11}, there exist $m >0$ and $n_0\in \mathbb{N}$ verifying
	\[
	\|u_n\|^{2} < m < 1,
	\,\,\,\text{for any}\,\,\,  n \geq n_0.
	\]
Choose $\alpha>4 \pi$ close to 4$\pi$, $t>1$ close to $1$ with $\alpha m t<4\pi$, and
by \cite[Lemma A.1]{Willem},
	there exists $\omega  \in L^{qt'}(B_R(0))$ such that $\vert u_{n}\vert\leq \omega$ a.e. in $B_R(0)$, where $t'$ is the conjugate exponent of $t$. Thus,
\begin{equation*}
	|u_n|^{q}(e^{\alpha  |u_n|^{2}}-1)\leq \omega^{q}(e^{\alpha  |u_n|^{2}}-1), \quad \text{a.e. in}\, B_R(0),
	\end{equation*}
and
\begin{equation}\label{neww1}
		|F(u_n)| \leq \zeta|u_n|^{\tau+1}+C\omega^{q}(e^{\alpha  |u_n|^{2}}-1), \quad \text{a.e. in}\, B_R(0).
	\end{equation}
Setting
	$$
	h_n(x)=C(e^{\alpha |u_n|^{2}}-1),
	$$
We can argue as in the proof of Lemma \ref{alphat11}, there exists $C>0$ such that
$$
	h_n \in L^{t}(\mathbb{R}^{2}) \quad \mbox{and} \quad \sup_{n \in \mathbb{N}}|h_n|_{t}<+\infty,
	$$
Therefore, for some subequence of $(u_n)$, still denoted by itself, we derive that
	\begin{equation} \label{Newlimit}
		h_n \rightharpoonup h=C(e^{\alpha |u|^{2}}-1), \,\, \mbox{in} \,\, L^{t}(\mathbb{R}^{2}).
	\end{equation}
\begin{claim} Now we show that
		$$
		 \omega^{q}h_n \to  \omega^{q} h \quad \mbox{in} \quad L^{1}(B_R(0)), \quad \forall R>0.
		$$
	\end{claim}
	\noindent Indeed, for each $R>0$, we consider the characteristic function $\chi_R$ associated with $B_R(0) \subset \mathbb{R}^{2}$, that is,
	$$
	\chi_R(x)=
	\left\{
	\begin{array}{l}
		1, \quad x \in B_{R}(0), \\
		0, \quad x \in \mathbb{R}^{N} \setminus B_{R}(0),
	\end{array}
	\right.
	$$
	which belongs to  $L^{qt'}(\mathbb{R}^{2})$. Thus, by the weak limit (\ref{Newlimit}),
	$$
	\int_{\mathbb{R}^{2}} \omega^{q}\chi_R h_n\,dx \to \int_{\mathbb{R}^{2}}\omega^{q}\chi_R h\,dx,
	$$
	or equivalently,
	\begin{equation} \label{Neww2}
	\int_{B_R(0)}\omega^{q} h_n\,dx \to \int_{B_R(0)}\omega^{q} h\,dx.
	\end{equation}
Hence, by $u_n \to u$ in $L^{\tau+1}(B_R)$,\,\,\eqref{neww1} and \eqref{Neww2}, applying a variant of the Lebesgue Dominated Convergence Theorem, we deduce that
$$
	F(u_n) \to F(u) \quad \mbox{in} \quad L^{1}(B_R(0)).
	$$

\end{proof}

The next lemma is crucial in our argument.

\begin{lemma} \label{convergencia} Let $(u_n) \subset H_{rad}^{1}(\mathbb{R}^{2})$ be a sequence  with $u_n \in S(a)$ and
	$$
	\limsup_{n \to +\infty} |\nabla u_n |_2^{2}  < 1-a^2.
	$$
	Then, there exists $\alpha$ close to $4\pi$, such that for all $q >2$,
	$$
	|u_n|^{q}(e^{\alpha  |u_n(x)|^{2}}-1) \to |u|^{q}(e^{\alpha |u(x)|^{2}}-1) \,\, \mbox{in} \,\, L^{1}(\mathbb{R}^N).
	$$
	
\end{lemma}
\begin{proof}
	Arguing as in Corollary \ref{Convergencia em limitados}, there are 	$\alpha>4\pi$ close to $4\pi$ and $t>1$ close to $1$ such that the sequence
	$$
	h_n(x)=(e^{\alpha |u_n(x)|^{2}}-1),
	$$
	is a bounded sequence in $L^{t}(\mathbb{R}^N)$. Therefore,  for some subsequence of $(h_n)$, still denoted by itself, we derive that
	\begin{equation*}
		h_n \rightharpoonup h=(e^{\alpha |u|^{2}}-1) \quad \mbox{in} \quad L^{t}(\mathbb{R}^{2}).
	\end{equation*}
	For $t'=\frac{t}{t-1}$, we know that the embedding $H_{rad}^{1}(\mathbb{R}^N) \hookrightarrow L^{qt'}(\mathbb{R}^N)$ is compact, then
	$$
	u_n \to u \quad \mbox{in} \quad L^{qt'}(\mathbb{R}^N),
	$$	
	and so,
	$$
	|u_n|^{q} \to  |u|^{q}  \quad \mbox{in} \quad L^{t'}(\mathbb{R}^N).
	$$
	Thus,
	$$
	\lim_{n \to +\infty}\int_{\mathbb{R}^{2}}|u_n|^{q}h_n(x)\,dx=\int_{\mathbb{R}^{2}}|u|^{q}h(x)\,dx,
	$$
	that is,
	$$
	\lim_{n \to +\infty}\int_{\mathbb{R}^{2}}|u_n|^{q}(e^{\alpha |u_n(x)|^{2}}-1)\,dx=\int_{\mathbb{R}^{2}}|u|^{q}(e^{\alpha |u(x)|^{2}}-1)\,dx.
	$$
	Since
	$$
	|u_n|^{q}(e^{\alpha |u_n(x)|^{2}}-1) \geq 0 \quad \mbox{and} \quad |u|^{q}(e^{\alpha |u(x)|^{2}}-1) \geq 0,
	$$
	the last limit gives
	$$
	|u_n|^{q}(e^{\alpha |u_n(x)|^{2}}-1) \to |u|^{q}(e^{\alpha |u(x)|^{2}}-1) \quad \mbox{in} \quad L^{1}(\mathbb{R}^2).
	$$
	
\end{proof}

\begin{corollary} \label{Convergencia em limitados1} Let $(u_{n})$ be a sequence in $H_{rad}^{1}(\mathbb{R}^{2})$ with $u_n \in S(a)$ and
	$$
	\limsup_{n \to +\infty} |\nabla u_n |_2^{2}  < 1-a^2.
	$$
	If $u_n \rightharpoonup u$ in $H^{1}(\mathbb{R}^{2})$ and $u_n(x) \to u(x)$ a.e in $\mathbb{R}^{2}$, then
	$$
	F(u_n) \to F(u) \quad \mbox{and} \quad f(u_n)u_n \to f(u)u \,\, \mbox{in} \,\, L^{1}(\mathbb{R}^2).
	$$ 	
\end{corollary}
\begin{proof}By \eqref{1.3},
	$$
	|F(t)|\leq\zeta |t|^{\tau+1}+C|t|^{q}(e^{\alpha t^{2}}-1)\,\, \text{ for all }\, t \in \mathbb{R},
	$$
	where $\alpha>4\pi$ close to $4\pi$ and $q>2$ as in the last Lemma \ref{convergencia}. Therefore,
	\begin{equation} \label{Domina1}
		|F(u_n)| \leq \zeta |u_n|^{\tau+1}+C|u_n|^{q}(e^{\alpha |u_n|^{2}}-1).
	\end{equation}
	By 	Lemma \ref{convergencia},
	$$
	|u_n|^{q}(e^{\alpha |u_n(x)|^{2}}-1) \to |u|^{q}(e^{\alpha |u(x)|^{2}}-1) \quad \mbox{in} \quad L^{1}(\mathbb{R}^2),
	$$
	and by the compact embedding $H_{rad}^{1}(\mathbb{R}^N) \hookrightarrow L^{\tau+1}(\mathbb{R}^N)$,
	$$
	u_n \to u \quad \mbox{in} \quad L^{\tau+1}(\mathbb{R}^2).
	$$
	Now, we can use the Lebesgue's Theorem to conclude that
	$$
	F(u_n) \to F(u) \quad \mbox{in} \quad L^{1}(\mathbb{R}^2).
	$$
	A similar argument works to show that
	$$
	f(u_n)u_n \to f(u)u \quad \mbox{in} \quad L^{1}(\mathbb{R}^2).
	$$
	
\end{proof}
From now on, we will use the same notations introduced in Section 2 to apply our variational procedure, more precisely

\begin{itemize}
	\item[\rm (1)] $S(a)=\{u \in H^{1}(\mathbb{R}^2)\,:\, | u |_2=a\, \}$ is the sphere of radius $a>0$ defined with the norm $|\,\,\,\,|_2$.

	\item[\rm (2)] $J: H^{1}(\mathbb{R}^2)\rightarrow \mathbb{R}$ with
	$$
	J(u)=\frac{1}{2}\int_{\mathbb{R}^{2}} |\nabla u|^2 dx-\int_{\mathbb{R}^{2}} F(u)dx.
	$$

	\item[\rm (3)] $\mathcal{H}: H\rightarrow \mathbb{R}$  with
	\begin{equation*}
		\mathcal{H}(u, s)(x)=e^{s}u(e^{s}x).
	\end{equation*}
	
	\item[\rm (4)] $\tilde{J}: H\rightarrow \mathbb{R}$ with
	$$
	\tilde{J}(u, s)=\frac{e^{2s}}{2}\int_{\mathbb{R}^{2}} |\nabla u|^2 dx-\frac{1}{e^{2s}}\int_{\mathbb{R}^{2}} F(e^{s}u(x))dx.
	$$
	
\end{itemize}

\subsection{The minimax approach}

We will  prove that $\tilde{J}$ on $S(a)\times \mathbb{R}$  possesses a kind of mountain-pass geometrical structure.

\begin{lemma}\label{vicente} Assume that $(f_1)-(f_2)$ and ($f_3$) hold and let $u\in S(a)$ be arbitrary but fixed. Then we have:\\
	\noindent (i)  $\vert \nabla \mathcal{H}(u, s) \vert_{2}\rightarrow 0$ and $J(\mathcal{H}(u, s))\rightarrow 0$ as $s\rightarrow -\infty$;\\
	\noindent (ii) $\vert \nabla \mathcal{H}(u, s) \vert_{2}\rightarrow +\infty$ and $J(\mathcal{H}(u, s))\rightarrow -\infty$ as $s\rightarrow +\infty$.
\end{lemma}

\begin{proof} \mbox{} By a straightforward calculation, it follows that
	\begin{equation} \label{CONV0}\int_{\mathbb{R}^2}\vert \mathcal{H}(u, s)(x)\vert^{2}\,dx=a^{2}, \quad \int_{\mathbb{R}^2}|\mathcal{H}(u, s)(x)|^{\xi}\,dx= e^{(\xi-2)s}\int_{\mathbb{R}^2}|u(x)|^{\xi}\,dx, \quad \forall \xi \geq 2,
	\end{equation}
	and
	\begin{equation} \label{CONV1*}
		\int_{\mathbb{R}^2}\vert \nabla \mathcal{H}(u, s)(x)\vert^{2}\,dx=e^{2s}\int_{\mathbb{R}^2}|\nabla u|^2 dx.
	\end{equation}
	From the above equalities, fixing $\xi>2$, we have
	\begin{equation} \label{CONV1}
		\vert \nabla \mathcal{H}(u, s) \vert_{2}\rightarrow 0 \quad \mbox{and} \quad |\mathcal{H}(u,s)|_{\xi} \to 0 \quad  \mbox{as} \quad s \to -\infty.
	\end{equation}
	Thus, there are $s_1<0$ and $m \in (0,1)$ such that
	$$
	\| \mathcal{H}(u, s)\|^{2} \leq m, \quad \forall s \in (-\infty_1,s_1].
	$$
By \eqref{1.3},
	$$
	|F(t)|\leq\zeta |t|^{\tau+1}+C|t|^{q}(e^{\alpha t^{2}}-1)\,\, \text{ for all }\, t \in \mathbb{R},
	$$
	where $\alpha>4\pi$ close to $4\pi$ and $q>2$ as in the last Lemma \ref{convergencia}. Hence,
	$$
	|F( \mathcal{H}(u, s))| \leq \zeta| \mathcal{H}(u, s)|^{\tau+1}+C| \mathcal{H}(u, s)|^{q}(e^{\alpha | \mathcal{H}(u, s)|^{2}}-1), \quad \forall s \in (-\infty,s_1].
	$$
	Using the H\"older's inequality together with Lemma \ref{Cao}, there exists $C=C(u,m)>0$
	such that
   $$
  \int_{\mathbb{R}^2} (e^{\alpha | \mathcal{H}(u, s)|^{2}}-1)^{t}dx\leq C,
   $$
   and so,
	$$
	\int_{\mathbb{R}^2}|F( \mathcal{H}(u, s))|\,dx \leq \zeta\int_{\mathbb{R}^2}| \mathcal{H}(u, s)|^{\tau+1}\,dx+C_1\Big(\int_{\mathbb{R}^2}| \mathcal{H}(u, s)|^{qt'}\,dx\Big)^{1/t'}, \quad \forall s \in (-\infty,s_1],
	$$
	where $t'=\frac{t}{t-1}$, and $t>1$ is close to 1.  Now, by using (\ref{CONV1}),
	$$
	\int_{\mathbb{R}^2}|F( \mathcal{H}(u, s))| \to 0 \quad \mbox{as} \quad s \to -\infty,
	$$
	from where it follows that
	$$
	J(\mathcal{H}(u, s))\rightarrow 0 \quad \mbox{as} \quad  s\rightarrow -\infty,
	$$
	showing $(i)$.
	
	In order to show $(ii)$, note that by (\ref{CONV1*}),
	$$
	|\nabla \mathcal{H}(u, s) |_{2}\rightarrow +\infty \quad \mbox{as} \quad s \to +\infty.
	$$
	On the other hand, by $(f_4)$,
	$$
	J(\mathcal{H}(u, s)) \leq \frac{1}{2}|\nabla \mathcal{H}(u,s)|_2^2-\frac{\mu}{p}\int_{\mathbb{R}^2}|\mathcal{H}(u,s)|^{p}\,dx=e^{2s}\int_{\mathbb{R}^2}|\nabla u|^2 dx-\frac{\mu e^{(p-2)s}}{p}\int_{\mathbb{R}^2}|u(x)|^{p}\,dx.
	$$
	Since $p>4$, the last inequality yields
	$$
	J(\mathcal{H}(u, s))\rightarrow -\infty \,\, \mbox{as} \,\, s\rightarrow +\infty.
	$$

\end{proof}

\begin{lemma}  \label{P1} Assume that $(f_1)-(f_3)$ hold.  Then there exists $K(a)>0$ small enough such that
	$$
	0<\sup_{u\in A} J(u)<\inf_{u\in B} J(u)
	$$
	with
	$$
	A=\left\{u\in S(a), \int_{\mathbb{R}^2} |\nabla u|^2 dx\leq K(a) \right\}\quad \mbox{and} \quad B=\left\{u\in S(a), \int_{\mathbb{R}^2} |\nabla u|^2 dx=2K(a) \right\}.
	$$
\end{lemma}
\begin{proof}
	We will need the following Gagliardo-Sobolev inequality: for any $\xi> 2$,
	$$
	|u|_\xi\leq C(\xi, 2)|\nabla u|_2^{\gamma}|u|_2^{1-\gamma},
	$$
	where $\gamma=2(\frac{1}{2}-\frac{1}{\xi})$. If we fix $K(a)<\frac{1-a^2}{2}$, $|\nabla u|^{2}_{2}\leq K(a)$ and $|\nabla v|^{2}_2=2K(a)$, the conditions $(f_1)-(f_2)$ combined with Lemma \ref{Cao} ensure that
	$$
	\int_{\mathbb{R}^2}F(u)\,dx \leq \zeta|u|^{\tau+1}_{\tau+1}+C_2|u|^{q}_{qt'}
	$$
	where $q>2$, $t'=\frac{t}{t-1}>1$, $t$ closed to 1. Then, by the Gagliardo-Sobolev inequality,
	$$
	\int_{\mathbb{R}^2}F(v)\,dx \leq C_1|\nabla v|_2^{\tau-1}+C_2|\nabla v|^{(q-2/t')}_{2}.
	$$
	From $(f_3)$, $F(u)\geq 0$ for any $u\in H^{1}(\mathbb{R}^{2})$, then
	\begin{eqnarray*}
		J(v)-J(u) &=&\frac{1}{2}\int_{\mathbb{R}^2}|\nabla v|^{2}\,dx-\frac{1}{2}\int_{\mathbb{R}^2}|\nabla u|^{2}\,dx-\int_{\mathbb{R}^2}F(v)\,dx+\int_{\mathbb{R}^2}F(u)\,dx\\
		&\geq &\frac{1}{2}\int_{\mathbb{R}^2}|\nabla v|^{2}\,dx-\frac{1}{2}\int_{\mathbb{R}^2}|\nabla u|^{2}\,dx-\int_{\mathbb{R}^2}F(v)\,dx,
	\end{eqnarray*}
	and so,
	$$
	J(v)-J(u) \geq \frac{1}{2}K(a)-C_3K(a)^{\frac{\tau-1}{2}}-C_4 K(a)^{(\frac{q}{2}-\frac{1}{t'})}.
	$$
	Since $\tau>3$ and $t'>0$ with $\frac{q}{2}-\frac{1}{t'}>1$, decreasing $K(a)$ if necessary, it follows that
	$$
	\frac{1}{2}K(a)-C_3K(a)^{\frac{\tau-1}{2}}-C_4 K(a)^{(\frac{q}{2}-\frac{1}{t'})}>0,
	$$
	showing the desired result.
\end{proof}

As a byproduct of the last lemma is the following corollary.
\begin{corollary} \label{newcor}
	There exists $K(a)>0$ small enough such that if $u \in S(a)$ and $|\nabla u|^2_{2}\leq K(a)$, then $J(u) >0$.
\end{corollary}
\begin{proof} Arguing as in the last lemma,
	$$
	J(u) \geq \frac{1}{2}|\nabla u|_{2}^{2}-C_1|\nabla u|_2^{\tau-1}-C_2|\nabla u|^{(q-\frac{2}{q^{*}})}_{2}>0,
	$$
	for $K(a)>0$ small enough.	
\end{proof}

In what follows, we fix $u_0 \in S(a)$ and apply Lemma \ref{vicente}  and Corollary  \ref{newcor}  to get two numbers $s_1<0$ and $s_2>0$, of such way that the functions $u_1=\mathcal{H}(u_0, s_1)$ and $u_2=\mathcal{H}(u_0, s_2)$ satisfy
$$
|\nabla u_1|^2_2<\frac{K(a)}{2},\,\,  |\nabla u_2|_2^2>2K(a),\,\, J(u_1)>0\,\,\mbox{and} \,\, J(u_2)<0.
$$

Now, following the ideas from Jeanjean \cite{jeanjean1}, we fix the following mountain pass level given by
$$
\gamma_\mu(a)=\inf_{h \in \Gamma}\max_{t \in [0,1]}J(h(t))
$$
where
$$
\Gamma=\left\{h \in C([0,1],S(a))\,:\,h(0)=u_1 \,\, \mbox{and} \,\, h(1)=u_2 \right\}.
$$
From Lemma \ref{P1},
$$
\max_{t \in [0,1]}J(h(t))>\max \left\{J(u_1),J(u_2)\right\}.
$$
\begin{lemma} \label{ESTMOUNTPASS} There holds $\displaystyle \lim_{\mu \to +\infty}\gamma_\mu(a)=0$.

\end{lemma}
\begin{proof} In what follow we set the path $h_0(t)=\mathcal{H}\big(u_0, (1-t)s_1+ts_2\big) \in \Gamma$. Then, by $(f_4)$,
	$$
	\gamma_\mu(a) \leq \max_{t \in [0,1]}J(h_0(t)) \leq \max_{r \geq 0}\left\{\frac{r^{2}}{2}|\nabla u_{0}|_{2}^{2}-\frac{\mu}{p}r^{p-2}|u_{0}|_{p}^{p}\right\}
	$$
	and so,
	$$
	\gamma_\mu(a) \leq C_2\left(\frac{1}{\mu}\right)^{\frac{2}{p-4}} \to 0 \quad \mbox{as} \quad \mu \to +\infty,
	$$
	for some $C_2>0$. Here, we have used the fact that $p>4$.
	
\end{proof}

Arguing as Section 2, in what follows $(u_n)$ denotes the $(PS)$ sequence associated with the level $\gamma_\mu(a)$, which satisfies:
\begin{equation} \label{gamma(a)1}
	J(u_n) \to \gamma_\mu(a), \,\, \mbox{as} \,\, n \to +\infty,
\end{equation}
\begin{equation} \label{EQ101}
	-\Delta u_n-f(u_n)=\lambda_nu_n\ + o_n(1), \,\, \mbox{in} \,\, (H^{1}(\mathbb{R}^2))^*,
\end{equation}
for some sequence $(\lambda_n) \subset \mathbb{R}$, and
\begin{equation} \label{EQ1**1}
	Q(u_n)=\int_{\mathbb{R}^2}|\nabla u_n|^2 dx +2 \int_{\mathbb{R}^2} F(u_n) dx- \int_{\mathbb{R}^2} f(u_n) u_n dx\to 0,\,\, \mbox{as} \,\, n \to +\infty.
\end{equation}
Moreover, $(u_n)$ is a bounded sequence, and so, the number $\lambda_n$  must satisfy the equality below
\begin{equation} \label{lambdan1}
	\lambda_n=\frac{1}{a^{2}}\left\{ |\nabla u_n|^{2}_{2} -\int_{\mathbb{R}^2} f(u_n) u_n dx \right\}+o_n(1).
\end{equation}

\begin{lemma}\label{newlem1} There holds
	$$
	\limsup_{n \to +\infty}\int_{\mathbb{R}^2}F(u_n)\,dx \leq \frac{2}{\theta-4}\gamma_\mu(a).
	$$
\end{lemma}
\begin{proof} Using the fact that $J(u_n)=\gamma_\mu(a)+o_n(1)$ and $Q(u_n)=o_n(1)$, it follows that
	$$
	2{J}(u_n)+Q(u_n)=2\gamma_\mu(a)+o_n(1),
	$$
	and so,
	\begin{eqnarray}\label{equa1}
	2|\nabla u_n|_{2}^{2}-\int_{\mathbb{R}^2}f(u_n)u_n\,dx=2\gamma_\mu(a)+o_n(1).
	\end{eqnarray}
Using that $J(u_n)=\gamma_\mu(a)+o_n(1)$, we get
$$
4\int_{\mathbb{R}^2} F(u_n) dx+4\gamma_\mu(a)+o_n(1)-\int_{\mathbb{R}^2}f(u_n)u_n\,dx=2\gamma_\mu(a)+o_n(1).
$$
Hence, by $(f_3)$,
$$
2\gamma_\mu(a)+o_n(1)=\int_{\mathbb{R}^2}f(u_n)u_n\,dx-4\int_{\mathbb{R}^2} F(u_n) dx\geq (\theta-4)\int_{\mathbb{R}^2} F(u_n) dx.
$$
Since $\theta>4$, we have
$$
	\limsup_{n \to +\infty}\int_{\mathbb{R}^2}F(u_n)\,dx \leq \frac{2}{\theta-4}\gamma_\mu(a).
$$
\end{proof}

\begin{lemma} \label{goodest} The sequence $(u_n)$ satisfies  $\displaystyle \limsup_{n \to +\infty}|\nabla u_n|_{2}^{2} \leq \frac{2(\theta-2)}{\theta-4} \gamma_\mu(a)$. Hence, there exists $\mu^*>0$ such that
	$$
	\limsup_{n \to +\infty}|\nabla u_n|_{2}^{2}<1-a^2, \,\,\text{for any} \,\,\,\mu \geq \mu^*.
	$$	
	
\end{lemma}
\begin{proof}  Since $J(u_n)=\gamma_\mu(a)+o_n(1)$, we have
	$$
	\int_{\mathbb{R}^2}|\nabla u_n|^2\,dx=2\gamma_\mu(a)+2\int_{\mathbb{R}^2}F(u_n)\,dx+o_n(1).
	$$
	Thereby, by Lemma \ref{newlem1},
$$
	\limsup_{n \to +\infty}|\nabla u_n|_{2}^{2}  \leq \frac{2(\theta-2)}{\theta-4} \gamma_\mu(a).
	$$
\end{proof}

\begin{lemma} \label{mu} Fix $\mu \geq \mu^*$, where $\mu^*$ is given in Lemma \ref{goodest}. Then,  $(\lambda_n)$ is a bounded sequence with

	$$
	\limsup_{n \to +\infty} |\lambda_n| \leq \frac{2\theta}{a^2(\theta-4)}\gamma_\mu(a) \quad \mbox{and} \quad  \limsup_{n \to +\infty} \lambda_n =-\frac{2}{a^2}\liminf_{n \to +\infty}\int_{\mathbb{R}^2}F(u_n)\,dx.
	$$
\end{lemma}
\begin{proof} The boundedness of $(u_n)$ yields that $(\lambda_n)$ is bounded, because
	$$
	\lambda_n|u_n|_{2}^{2} =|\nabla u_n|_{2}^{2}-\int_{\mathbb{R}^2}f(u_n)u_n\,dx+o_n(1),
	$$
	and as $|u_n|_{2}^{2}=a^2$, we have
	$$
	\lambda_na^2=|\nabla u_n|_{2}^{2}-\int_{\mathbb{R}^2}f(u_n)u_n\,dx+o_n(1).
	$$
	Hence,
	$$
	|\lambda_n|a^2 \leq |\nabla u_n|_{2}^{2}+\int_{\mathbb{R}^2}f(u_n)u_n\,dx+o_n(1).
	$$

The limit (\ref{EQ1**1}) together with Lemmas \ref{newlem1} and \ref{goodest} ensures that $(\int_{\mathbb{R}^2}f(u_n)u_n\,dx)$ is bounded with
$$
\limsup_{n \to +\infty}\int_{\mathbb{R}^2}f(u_n)u_n\,dx\leq \frac{2\theta}{\theta-4} \gamma_\mu(a).
$$
This is enough to conclude that $(\lambda_n)$  is a bounded sequence with
$$
\limsup_{n \to +\infty} |\lambda_n| \leq \frac{2\theta}{a^2(\theta-4)}\gamma_\mu(a).
$$
In order to prove the second inequality, the equality
	$$
	\lambda_na^2=|\nabla u_n|_{2}^{2}-\int_{\mathbb{R}^2}f(u_n)u_n\,dx+o_n(1)
	$$
	together with the  limit (\ref{EQ1**1}) leads to
	$$
	\lambda_na^2=-2\int_{\mathbb{R}^2}F(u_n)\,dx+o_n(1),
	$$
	showing the desired result.
\end{proof}

Now, we restrict our study to the space $H_{rad}^{1}(\mathbb{R}^2)$. For any $\mu \geq \mu^*$, using Lemmas \ref{convergencia} and \ref{goodest}, Corollary \ref{Convergencia em limitados1}, it follows that
$$
\lim_{n \to +\infty}\int_{\mathbb{R}^2}f(u_n)u_n\,dx=\int_{\mathbb{R}^2}f(u)u\,dx,
$$
and
$$
\lim_{n \to +\infty}\int_{\mathbb{R}^2}F(u_n)\,dx=\int_{\mathbb{R}^2}F(u)\,dx,
$$
where $u_n \rightharpoonup u$ in $H^{1}(\mathbb{R}^2)$. The last limit implies that $u \not=0$, because otherwise, Corollary \ref{Convergencia em limitados1} gives
$$
\lim_{n \to +\infty}\int_{\mathbb{R}^2}F(u_n)\,dx=\lim_{n \to +\infty}\int_{\mathbb{R}^2}f(u_n)u_n\,dx=0,
$$
and by Lemma \ref{mu},
$$
\limsup_{n \to +\infty} \lambda_n \leq 0.
$$

Since $(u_n)$ is bounded in $H^{1}(\mathbb{R}^2)$ and $\displaystyle \limsup_{n \to +\infty}|\nabla u_n|_{2}^{2}<1-a^2$ if $\mu \geq \mu^*$, Corollary \ref{Convergencia em limitados1}  together with $(f_1)-(f_2)$ and the equality below
\begin{equation*}
	\lambda_n| u_n|_{2}^{2}=|\nabla u_n|^{2}_{2}-\int_{\mathbb{R}^2}f(u_n)u_n\,dx+o_n(1),
\end{equation*}
lead to
\begin{equation}  \label{mu2}
	\lambda_na^2=|\nabla u_n|_{2}^{2}+o_n(1).
\end{equation}
From this,
$$
0 \geq \limsup_{n \to +\infty} \lambda_n a^2= \limsup_{n \to +\infty} |\nabla u_n|^{2}_{2} \geq \liminf_{n \to +\infty} |\nabla u_n|^{2}_{2}\geq 0,
$$
then
$$
|\nabla u_n|^{2}_{2} \to 0,
$$
which is absurd, because $\gamma_\mu(a)>0$.

\section{Proof of Theorem \ref{T2}}

The above analysis ensures that the weak limit $u$ of $(u_n)$ is nontrivial.  Moreover, the equality
$$
\limsup_{n \to +\infty} \lambda_n =-\frac{2}{a^2}\liminf_{n \to +\infty}\int_{\mathbb{R}^2}F(u_n)\,dx
$$
ensures that
$$
\limsup_{n \to +\infty} \lambda_n =-\frac{2}{a^2}\liminf_{n \to +\infty}\int_{\mathbb{R}^2}F(u)\,dx<0.
$$
From this, for some subsequence, still denoted by $(\lambda_n)$, we can assume  that
$$
\lambda_n \to \lambda_a<0 \quad \mbox{as} \quad n \to +\infty.
$$
Now, the equality (\ref{EQ101})
implies that
\begin{equation} \label{EQ21}
	-\Delta u-f(u)=\lambda_au \quad \mbox{in} \quad \mathbb{R}^2.
\end{equation}
Thus,
$$
|\nabla u|_{2}^{2}-\lambda_a|u|_{2}^{2}=\int_{\mathbb{R}^2}f(u)u\,dx.
$$
On the other hand,
$$
|\nabla u_n|_{2}^{2}-\lambda_n|u_n|_{2}^{2}=\int_{\mathbb{R}^2}f(u_n)u_n\,dx+o_n(1),
$$
or yet,
$$
|\nabla u_n|_{2}^{2}-\lambda_a|u_n|_{2}^{2}=\int_{\mathbb{R}^2}f(u_n)u_n\,dx+o_n(1).
$$
Recalling that
$$
\lim_{n \to +\infty}\int_{\mathbb{R}^2}f(u_n)u_n\,dx=\int_{\mathbb{R}^2}f(u)u\,dx,
$$
we derive that
$$
\lim_{n \to +\infty}(|\nabla u_n|_{2}^{2}-\lambda_a|u_n|_{2}^{2})=|\nabla u|_{2}^{2}-\lambda_a|u|_{2}^{2}.
$$
Since $\lambda_a<0$, the last limit implies that
$$
u_n \to u \quad \mbox{in} \quad H^{1}(\mathbb{R}^2),
$$
implying that $|u|_{2}^{2}=a$. This establishes the desired result.\\

\noindent \textsc{Claudianor O. Alves } \\
Unidade Acad\^{e}mica de Matem\'atica\\
Universidade Federal de Campina Grande \\
Campina Grande, PB, CEP:58429-900, Brazil \\
\texttt{coalves@mat.ufcg.edu.br} \\
\noindent and \\
\noindent \textsc{Chao Ji} \\
Department of Mathematics\\
East China University of Science and Technology \\
Shanghai 200237, PR China \\
\texttt{jichao@ecust.edu.cn}\\
\noindent and \\
\noindent \textsc{Ol\'{i}mpio Hiroshi, Miyagaki} \\
Departamento de Matem\'{a}tica \\
Universidade Federal de S\~ao Carlos \\
S\~ao Carlos,  SP, CEP:13565-905,  Brazil\\
\texttt{olimpio@ufscar.br}

\end{document}